\documentclass[11pt]{article}

\usepackage{latexsym}
\usepackage{graphicx} 
\usepackage{booktabs} 
\usepackage{multirow}
\usepackage{multicol}
\usepackage[table]{xcolor} 
\usepackage{array} 
\usepackage{ragged2e} 
\usepackage{amsmath}
\usepackage{color}
\usepackage{xcolor}
\usepackage{amsfonts}
\usepackage{amsthm}
\usepackage{natbib}
\usepackage{soul}
\usepackage{hyperref}
\usepackage{enumerate}
\usepackage{csquotes}
\usepackage{subcaption}    
\usepackage{tikz} 
\definecolor{red}{HTML}{DC2626}  
\definecolor{blue}{HTML}{2563EB} 
\definecolor{grey}{HTML}{A1A1AA}
\usepackage{stackengine} 
\usepackage{varwidth}   
\usetikzlibrary{fit, positioning, backgrounds}
\usepackage{setspace}

\newtheorem{theorem}{Theorem}

\newtheorem{lemma}[theorem]{Lemma}
\newtheorem{proposition}[theorem]{Proposition}

\begin{document}

\title{Alignment Games}

\author{
  Pedro Cesar Lopes Gerum\thanks{Cleveland State University, \texttt{p.lopesgerum@csuohio.edu}} 
  \and 
  Thomas Lidbetter\thanks{Rutgers Business School, \texttt{tlidbetter@business.rutgers.edu} (Corresponding author)}
}

\date{} 

\maketitle
        
\begin{abstract}

\noindent This paper introduces {\em alignment games}, a new class of zero-sum games modeling strategic interventions where effectiveness depends on alignment with an underlying hidden state. Motivated by operational problems in medical diagnostics, economic sanctions, and resource allocation, this framework features two players, a Hider and a Searcher, who choose subsets of a given space. Payoffs are determined by their misalignment (symmetric difference), explicitly modeling the trade-off between commission errors (unnecessary action) and omission errors (missed targets), given by a cost function and a penalty function, respectively.

We provide a comprehensive theoretical analysis, deriving closed-form equilibrium solutions that contain interesting mathematical properties based on the game's payoff structure. When {\em cost and penalty functions} are unequal, optimal strategies are consistently governed by cost-penalty ratios. On the unit circle, optimal arc lengths are direct functions of this ratio, and in discrete games, optimal choice probabilities are proportional to element-specific ratios.

When costs are equal, the solutions exhibit rich structural properties and sharp threshold behaviors. On the unit interval, this manifests as a geometric pattern of minimal covering versus maximal non-overlapping strategies. In discrete games with cardinality constraints, play concentrates on the highest-cost locations, with solutions changing discontinuously as parameters cross critical values.

Our framework extends the theory of geometric and search games and is general enough that classical models, such as Matching Pennies, emerge as special cases. These results provide a new theoretical foundation for analyzing the strategic tension between comprehensive coverage and precise targeting under uncertainty.
\end{abstract}
        

\section{Introduction}

Decision-makers often confront a fundamental tension when the true state of a system remains hidden: intervention carries the risk of causing harm where it is not needed, while inaction may allow problems to worsen. In medical diagnosis, unnecessary biopsies cause patient suffering and complications, while missed diagnoses permit diseases to progress. In ecological management, pesticides applied to uninfested areas poison native species, while untreated invasions spread unchecked. In international relations, economic sanctions misapplied to compliant actors damage diplomatic relationships and harm civilian populations, while failing to sanction violators allows prohibited activities to continue. This trade-off between commission errors (harmful intervention) and omission errors (harmful inaction) appears wherever hidden states govern optimal interventions.

This paper introduces {\em alignment games}, a class of zero-sum games that captures this strategic tension. In an alignment game, a Searcher chooses a subset $S$ of a ground set $Q$, seeking to match the subset $H$ chosen by a Hider. The payoff to the Hider (which the Searcher minimizes) depends on both types of misalignment: commission errors $S \setminus H$ and omission errors $H \setminus S$. The payoff of an alignment game takes the form $P(H, S) = C(S \setminus H) + \Pi(H \setminus S)$, where $C$ measures the cost of commission errors and $\Pi$ the penalty for omission errors. The Searcher seeks to align intervention with reality; the Hider benefits from misalignment. The Hider may be a malicious adversary, or a stand-in for ``Nature'', which conspires to present a worst-case scenario for the Searcher.

The structure of alignment problems varies fundamentally with the domain. In continuous settings, players choose subsets of intervals or regions: border patrol must decide where to establish checkpoints along a frontier, creating congestion and economic disruption at chosen locations, while smugglers select crossing points; conservation teams apply controlled fires to forest sections, destroying habitats where fires were unnecessary, while areas requiring fire management remain vulnerable to catastrophic wildfires. Here, strategy spaces {\em have infinite cardinality}, and optimal play often involves probability densities over the continuum.

In discrete settings, players select from finite collections: fact-checkers targeting demographic groups risk introducing misinformation to previously unexposed audiences (where corrections can backfire and create new believers), while existing misinformation continues spreading in uncorrected populations; public health officials implementing quarantine zones disrupt communities and economies in designated areas, while disease spreads in uncontrolled regions. These discrete games yield different mathematical structures, where optimal strategies may concentrate on particular elements or spread uniformly across the ground set. Of course, our models are an oversimplification of the complex reality of the examples described above, but they serve as motivating problems where, in each case, the decision maker may be considered to be playing a game against Nature or against an adversary.

Alignment games extend established frameworks in game theory. \cite{ruckle1983geometric} introduced {\em geometric games} where two players choose subsets of a set, with payoffs determined by the chosen sets and their intersection.  In the words of \cite{ruckle1983geometric}:

\begin{displayquote} Two antagonists, known hereafter as RED and BLUE, choose [as pure strategies] subsets $R$ and $B$ respectively of a set $S$. BLUE then receives from RED a payoff which is a function of the triple $R$, $B$, and $R \cap B$ \ldots In general, RED and BLUE may not choose any subset of S, but rather RED must select from a collection $\mathcal{R}$ of admissible pure strategies for RED and BLUE from a collection $\mathcal{B}$. \end{displayquote}

Other works on geometric games include \cite{baston1989number}, \cite{zoroa2003raid} and \cite{zoroa1999generalization}. Unlike traditional geometric games that reward successful ``hits'' (intersections), alignment games penalize both types of misalignment through the symmetric difference of the players' sets. The study of geometric games led to the development of {\em accumulation games} \citep{alpern2014accumulation, alpern2010ruckle, kikuta2002continuous} and {\em caching games} \citep{alpern2012search, csoka2016solution, janosik2025caching}, where a Searcher seeks to accumulate a minimum threshold of resources secreted by a Hider.

All the classes of games mentioned above may be positioned in the wider field of {\em search games}, pioneered by \cite{bram19632} and \cite{Isaacs-Book-1965}, where a Searcher attempts to locate a Hider or hidden targets. However, search games typically minimize detection time or search cost, while alignment games balance the dual harms of misplaced intervention and missed opportunities. See \cite{alpern&gal03book}, \cite{garnaev2012search} and \cite{hohzaki2016search} for overviews of the field of search games.

There is also a relation to {\em inspection games} \citep{avenhaus1996inspection, avenhaus2002inspection, hohzaki2011inspection, von2016recursive} and {\em patrolling games} \citep{alpern2019optimizing, bui2023optimal,garrec2019continuous}. Inspection games focus on violation detection in regulatory contexts and patrolling games are concerned with detected an unwelcome intrusion, whereas alignment games address the broader challenge of matching interventions to hidden states where both action and inaction carry consequences.

The commission/omission framework is analogous to the concept of Type I and Type II error trade-off from hypothesis testing, albeit in a game-theoretic setting. Moreover, alignment games generalize the classic game of Matching Pennies. Where Matching Pennies rewards an exact match between binary choices, our framework extends to arbitrary set selections with payoffs that are a function of the degree of misalignment. Section~\ref{sec:equal} makes this connection precise, showing how Matching Pennies emerges as a special case when both players choose single elements from a two-element set.

This paper derives closed-form equilibrium solutions (min-max and max-min mixed strategies) for alignment games across multiple settings. For continuous domains, we solve games on the circle and unit interval under various cost structures. For discrete domains, we analyze games where players choose from power sets, face cardinality constraints, or select single elements. Our analysis reveals rich strategic behavior: threshold effects where optimal strategies shift discontinuously with parameter changes, symmetric equilibria that break into asymmetric ones as costs diverge, and connections between the geometric structure of the ground set and the support of optimal mixed strategies.

In Section~\ref{sec:intro}, we will define alignment games formally. In Section~\ref{sec:continuous} we will introduce some alignment games played in continuous domains: in particular, the unit interval or the circle. We give solutions of different variants of the game in both these domains. We move to discrete games in Section~\ref{sec:discrete}, starting with the case that the cost and penalty functions are different. We give a solution for the case that both players' strategy sets contain all possible subsets of the ground set. We then consider the special case of equal cost and penalty functions in Section~\ref{sec:equal}, and give solutions in the case that the one player can choose any subset of the ground set and the other player must choose a subset of a fixed cardinality. We also solve the case where both players may choose a single element of the ground set.

\section{Game Definition} \label{sec:intro}

We consider a family of zero-sum games between a Hider and a Searcher. The game is played on a set $Q$ equipped with two set functions $C, \Pi:2^Q \rightarrow \mathbb{R}$. We refer to $C$ as the  {\em cost function} and to $\Pi$ as the {\em penalty function}. The Hider's pure strategy set is a collection $\mathcal{H} \subseteq 2^Q$ of subsets of $Q$, and the Searcher's pure strategy set is some other collection $\mathcal{S} \subseteq 2^Q$ of subsets.

For a strategy pair $(H,S) \in \mathcal{H} \times \mathcal{S}$, the payoff function $P$ is given as follows.
\[
P(H,S) = C(S \setminus H) + \Pi(H \setminus S).
\]
The Searcher is the minimizer and the Hider is the maximizer. The interpretation is that the Searcher pays a cost for searching locations where she does not find anything and she pays a penalty for not searching locations that do contain something.  

In some cases, we will make the simplifying assumption that the cost and penalty functions are the same, in which case the payoff is simply equal to $C(H \triangle S)$, where $H \triangle S$ is the symmetric difference $(H \setminus S) \cup (S \setminus H)$ of $H$ and $S$.

A mixed strategy is given by a randomized choice of pure strategies. For mixed strategies $h$ and $s$ of the Hider and Searcher, respectively, we denote the expected payoff when these strategies are played by $P(h,s)$.

We seek min-max and max-min strategies for the players. More precisely, a mixed strategy $s^*$ is a min-max strategy for the Searcher if it minimizes $\sup_{H \in \mathcal{H}} P(H,s)$ over mixed strategies $s$. Similarly, a mixed strategy $h^*$ is a max-min strategy for the Hider if it maximizes $\inf_{S \in \mathcal{S}} P(h,S)$ over mixed strategies $h$. We say strategies $h^*$ and $s^*$ are {\em optimal} if 
\[
\sup_{H \in \mathcal{H}} P(H,s^*) = \inf_{S \in \mathcal{S}} P(h^*,S).
\]
In this case, we refer to the expected payoff in the equation above as the {\em value} of the game.

\section{Alignment Games in Continuous Domains}
\label{sec:continuous}

We begin our analysis by examining the alignment game played on continuous domains, starting with the unit circle in Subsection~\ref{sec:circle}, followed by the unit interval in Subsection~\ref{sec:interval}.

\subsection{Circle} \label{sec:circle}

We consider the alignment game played on a unit length circle, which we index by the interval $Q=[0,1]$, where the points $0$ and $1$ are identified. We will first consider the case that the Hider's strategy set $\mathcal{H}$ consists of all arcs of the circle of length $\alpha$, and the Searcher's strategy set and $\mathcal{S}$ consists of all arcs of length $\beta$. Thus, a strategy for the Searcher may be specified by a point $x$ in $[0,1]$, which is interpreted as the arc $[x, x+\beta]$, modulo $1$, and a strategy for the Hider is specified by a point $y$ in $[0,1]$, interpreted as the arc $[y, y+\alpha]$, modulo 1. Later we will allow one or both players to choose the arc lengths. 

For an arc $I$ of length $L$ of the circle, the cost and penalty functions are taken to be $C(I) = c L$ and $ \Pi(I) = \pi L$, where $c$ and $\pi$ are positive constants. In other words, the cost is proportional to the length of the ``wasted'' part of the Searcher's arc and the penalty is proportional to the length of the ``undiscovered'' part of the Hider's arc.

It is clear from the symmetry of the game that each player may choose the starting point of their arc uniformly at random on $[0,1]$. Hence, if the arc lengths are fixed, this uniform strategy is trivially optimal for both players. If the arcs lengths can be chosen by the players, then a strategy for that player is specified by that chosen length.

\begin{lemma}
\label{lemma:exp_payoff}
   If the Hider chooses an arc of length $\alpha$ with a starting point chosen uniformly at random and the Searcher chooses an arc of length $\beta$ with a starting point chosen uniformly at random, then the expected payoff is
\begin{align} \label{eq:circle}
    P(\alpha,\beta)  = \beta (1 - \alpha) c + \alpha (1 - \beta) \pi.
\end{align}
This is the value of the alignment game on a unit circle for the case that the Hider's and Searcher's arc lengths are fixed to be $\alpha$ and $\beta$, respectively.
 
\end{lemma}

\begin{proof}
Let $X$ be a point on the circle, chosen uniformly at random. Let $E_1$ be the event that $X$ lies in the Hider's arc $A$ and let $E_2$ be the event that $X$ lies in the Searcher's arc~$B$. Since the Hider's and Searcher's starting points are chosen uniformly at random, the probability $Pr(E_1)$ of the event $E_1$ is equal to $\alpha$ and $Pr(E_2) = \beta$. As starting points are chosen independently, $E_1$ and $E_2$ are independent.

Since the length of an arc on the circle corresponds to the probability of a randomly chosen point being contained in that arc, we can express the expected lengths of the non-overlapping regions in terms of these probabilities.

The expected cost of $A \setminus B$ is
\begin{align*} 
\mathbb E[C(A \setminus B)] &= c \cdot Pr(X \in B \setminus A) = c \cdot Pr(E_1^c \cap E_2) \\ &= c \cdot Pr(E_1^c)Pr(E_2) = c \beta(1-\alpha). \end{align*}

Similarly, the expected penalty of $B \setminus A$ is $\pi \alpha(1-\beta)$. Hence, the expected payoff is
\begin{align*} P(\alpha, \beta) &= \mathbb E[C(B \setminus A) + \Pi (A \setminus B) ] = c \beta(1-\alpha) + \pi \alpha(1-\beta).
\end{align*}
\end{proof}

We now consider variations of the game where one or both of the players are allowed to choose the length of their arc, and give solutions to those games in each case.

\begin{proposition}
\label{theorem:circle}
For the alignment game played on the unit circle,
\begin{enumerate}[(i)]
\item When both players can choose their arc lengths strategically, it is optimal for the Hider to choose $\alpha^* \equiv \frac{c}{c+\pi}$ and for the Searcher to choose $\beta^* \equiv \frac{\pi}{c+\pi}$, with both selecting starting points uniformly at random. The equilibrium payoff is $ \frac{c\pi}{c+\pi}$.

\item In the case where the Hider has a fixed arc length $\alpha$ and the Searcher can choose her arc length, the Searcher's optimal strategy depends on the size of $\alpha$, relative to $\alpha^*$. If $\alpha < \alpha^*$, it is optimal for the Searcher to choose $\beta = 0$, and the value of the game is $V = \alpha \pi$. If $\alpha \ge \alpha^*$, it is optimal for the Searcher to choose $\beta = 1$, and the value of the game is $(1-\alpha) c$.

\item In the case where the Searcher has a fixed arc length $\beta$ and the Hider can choose his arc length, optimal strategies are as follows. If $\beta < \beta^*$, it is optimal for the Hider to choose $\alpha=1$, and the value of the game is $(1-\beta)\pi$; if $\beta \ge \beta^*$, it is optimal for the Hider to choose $\alpha=0$, and the value of the game is $\beta c$.
\end{enumerate}
\end{proposition}

\begin{proof}
Beginning with part (i), let us substitute the Searcher's proposed strategy  $\beta = \beta^* $ into the expected payoff function~(\ref{eq:circle}).
\[
P(\alpha, \beta^*) = \alpha \left(1 - \frac{\pi}{c+\pi}\right) \pi + \frac{\pi}{c+\pi} \left(1 - \alpha\right) c = \frac{c\pi}{c+\pi}.
\]
Thus, the Searcher can ensure a payoff of at most $c \pi/(c + \pi)$.

It is also easy to show that if the Hider chooses $\alpha^* = \frac{c}{c+\pi}$, the expected payoff becomes $c \pi/(c+\pi)$.
Thus, the Hider can ensure the payoff is at least $c \pi/(c+\pi)$, and this must be the value of the game.

For part (ii), with fixed Hider arc length $\alpha$, we can rearrange the payoff function:
\begin{align*}
P(\alpha, \beta) = \alpha \pi - \alpha\beta\pi + \beta c - \alpha\beta c = \alpha \pi + \beta[c - \alpha(c + \pi)].
\end{align*}

The coefficient of $\beta$, which is $[c - \alpha(c + \pi)]$, determines the Searcher's best response. If $\alpha < \alpha^*$, this coefficient is positive, so the Searcher minimizes payoff with $\beta = 0$. If $\alpha > \alpha^*$, the coefficient is negative, so the best response becomes $\beta^* = 1$. If $\alpha = \alpha^*$, then any strategy of the Searcher is a best response, and in particular the strategy $\beta=1$.

The equilibrium payoffs follow directly from substituting these optimal values into the payoff function. 

Part (iii) is analogous to part (ii), and we omit the proof.
\end{proof}

Part (i) of Proposition~\ref{theorem:circle}, when players can choose the length of their intervals, is the most interesting. The optimal strategies for both players depends on the relative cost to the Searcher of searching in the wrong place, versus not searching in the right place. If the former cost is more significant, then the Searcher prefers a cautious strategy, and chooses a relatively short interval, whereas the Hider uses a more bold strategy and chooses a relatively long interval. If the relative size  of the costs are reversed, so are the optimal strategies.

It is also worth noting that in case (ii), if $\alpha=\alpha^*$, any strategy is optimal for the Searcher, as she is indifferent between all her strategies. Similarly in case (iii), if $\beta=\beta^*$.

\subsection{Unit Interval} \label{sec:interval}

We now consider the alignment game played on the unit interval. As in the previous subsection, the Hider and Searcher choose subintervals, this time of the unit interval $[0,1]$. Again, we may consider different variations of the game, depending on whether the Hider and Searcher are allowed to choose the length of their subinterval. Unlike the circle case, we cannot argue that the players should use uniform strategies, because of the endpoints of the interval. This complicates the analysis, and so we make the simplifying assumption $\pi = c=1$, so that the cost and penalty functions for an interval $I$ of length $L$ are both given by $C(I)=\Pi(I)=L$. The payoff is then the total length of the symmetric difference of the two players' chosen intervals.


It turns out that the case where the players can choose the length of their subintervals is simplest to solve. 

\begin{proposition}
    For the alignment game played on the unit with equal cost and penalty functions, where both players may choose the length of their subinterval, an optimal strategy for both players is to choose with equal probability the intervals $[0,1/2]$ and $[1/2,1]$. The value of the game is $1/2$.
\end{proposition}
\begin{proof}
Suppose one player plays the proposed strategy and the other player chooses an interval $I$ such that $|I \cap [0,1/2]|=x$ and $|I \cap[1/2,1]|=y$. Then the expected payoff is 
\[
1/2 \cdot ((1/2-x) + y) + 1/2 \cdot(x + (1/2-y)) = 1/2.
\]
Therefore, each player can ensure an expected payoff of precisely $1/2$, so this strategy is optimal for both players.
\end{proof}

We now turn to the case where one player's interval length is predetermined, and the other player can choose it.

\begin{theorem} \label{thm:interval2}
Consider the alignment game on the unit interval with equal cost and penalty functions. There are two cases where only one of the players can choose their interval length:
\begin{enumerate}[(i)]
\item First suppose the Hider's interval length is fixed at $\alpha \in [0,1]$ and the Searcher can choose her interval length. Then it is optimal for the Hider to choose with equal probability one of the intervals $[0,\alpha]$ and $[1-\alpha,1]$. If $\alpha \ge 1/2$, it is optimal for the Searcher to choose the whole interval $[0,1]$, and the value of the game is $1-\alpha$; if $\alpha \le 1/2$, it is optimal for the Searcher to choose the empty set, and the value of the game is $\alpha$.
\item Now suppose the Searcher's interval length is fixed at $\beta \in [0,1]$ and the Hider can choose his interval length. Then it is optimal for the Searcher to choose with equal probability one of the intervals $[0,\beta]$ and $[1-\beta,1]$. If $\beta \ge 1/2$, it is optimal for the Hider to choose the empty set, and the value of the game is $\beta$; if $\beta \le 1/2$, it is optimal for the Hider to choose the whole interval $[0,1]$, and the value of the game is $1-\beta$.
\end{enumerate}
\end{theorem}
\begin{proof}
For case (i), first suppose $\alpha \ge 1/2$. We show that the proposed Hider strategy ensures a payoff of at least $1-\alpha$. Consider a best response $I$ of the Searcher against this strategy. Clearly $I$ should contain $[1-\alpha,\alpha]$, since the Hider's interval contains this set with probability 1. Let $x$ be the length of $I \cap [0,1-\alpha]$ and let $y$ be the length of $I \cap [\alpha,1]$.  Then the expected payoff against the Hider's proposed strategy is 
\[
\frac{1}{2}(x+(1-\alpha-y)) + \frac{1}{2}(y + (1-\alpha-x)) = 1-\alpha.
\]
Hence, the value is at least $1-\alpha$. 

Also, it is clear that against the Searcher's proposed strategy of choosing the whole interval $[0,1]$, any Hider strategy results in a payoff of $1-\alpha$. Hence, the value is at most $1-\alpha$, and we have equality.

Next, suppose $\alpha \le 1/2$. In this case, a best response must be of the form $[x,y]$, where either $y \le \alpha$, or $x \le \alpha$ and $y \ge 1-\alpha$, or $x \ge 1-\alpha$. By symmetry, we need only consider the first two cases. In the first case, the expected payoff is $(1/2)(\alpha-(y-x))+(1/2)(\alpha+(y-x))=\alpha$. In the second case, the expected payoff is $(1/2)(y-(\alpha-x))+(1/2)(1-x-(y-(1-\alpha)))=1-\alpha \ge \alpha$. Hence, the Hider can ensure a payoff of at at least $\alpha$, so the value is at least $\alpha$.

The Searcher's proposed strategy also clearly ensures a payoff of precisely $\alpha$ against any Hider strategy, so the value is equal to $\alpha$.

The proof of part (ii) of the theorem is very similar to part (i), and we leave it as an exercise to the reader.
\end{proof}

The final remaining variant of the game is when both players' intervals are fixed. This is the most complex variant, and we present the solution in Theorem~\ref{thm:interval} for the case $\alpha=\beta$. To simplify the notation, we identify a strategy for either of the players with the leftmost point of their chosen interval, so that a point $a \in [0,1-\alpha]$ corresponds to the interval $[a,a+\alpha]$. If the Searcher chooses a point $a \in [0,1-\alpha]$ and the Hider chooses a point $b \in [0,1-\alpha]$, the payoff, which we denote $P(a,b)$, is given by the length of the symmetric difference $[a,a+\alpha] \triangle [b,b+\alpha]$, which is
\[
P(a,b) = \begin{cases}
2|a - b| & \text{if } |a - b| \leq \alpha \\
2\alpha & \text{if } |a - b| > \alpha.
\end{cases}
\]

\begin{theorem}
\label{thm:interval}
Consider the alignment game with equal cost and penalty functions, where the Hider's and Searcher's interval lengths are fixed at $\alpha \in [0,1]$. The optimal strategies and value of the game are as follows.

Define $M = \lfloor \frac{1}{\alpha} \rfloor$. It is optimal for the Hider to choose with equal probability each of the $M+1$ leftmost points $$a_i=\left(\frac{1-\alpha}{M}\right)i,~i=0,1,\ldots,M.$$
It is optimal for the Searcher to choose with equal probability each of the $M$ leftmost points $$b_i=\left(\frac{1+\alpha}{M+1}\right)i-\alpha,~i=1,\ldots,M.$$

The value $v$ of the game is given by
\[
v \equiv \frac{2\alpha(M^2-M-1)+2}{M(M+1)}.
\]
\end{theorem}

Before giving a proof of the theorem, we note that the theorem implies that the Hider will mix as few  intervals as possible that cover the whole interval $[0,1]$. The Searcher will mix among as many intervals as possible without overlapping. 

We depict the players' optimal strategies in~Figure~\ref{fig:game1} for $\alpha=0.4$ (so that $M=2$) and in~Figure~\ref{fig:game2} $\alpha=0.3$ (so that $M=3$). In each case, $p$ is the probability each of the pure strategies depicted is chosen by the player.
  
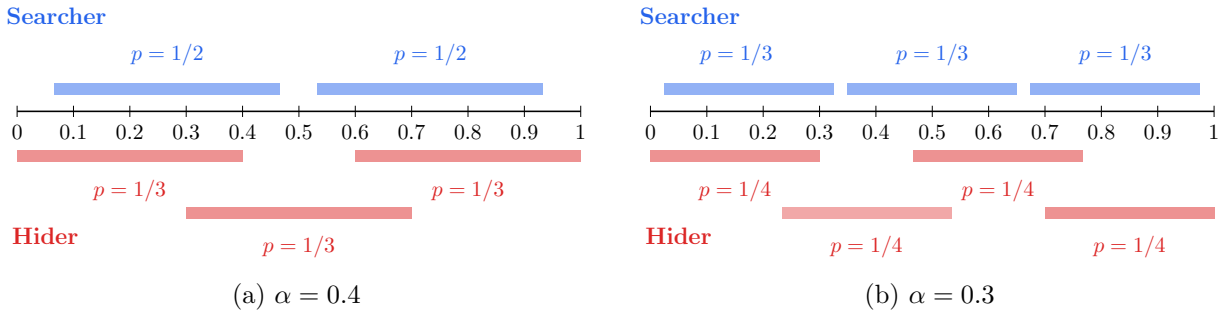
\begin{figure}[htbp]
    \centering
    \begin{subfigure}[b]{0.49\textwidth}
        \centering
        \resizebox{\textwidth}{!}{
            \begin{tikzpicture}
            \draw[thick] (0,0) -- (10,0);
            \foreach \x/\label in {0/0, 1/0.1, 2/0.2, 3/0.3, 4/0.4, 5/0.5, 6/0.6, 7/0.7, 8/0.8, 9/0.9, 10/1}
                \draw (\x,0.1) -- (\x,-0.1) node[below] {\label};
            \fill[red, opacity=0.5] (0,-0.9) rectangle (4,-0.7);
            \node[red] at (2,-1.4) {$p=1/3$};
            \fill[red, opacity=0.5] (3,-1.9) rectangle (7,-1.7);
            \node[red] at (5,-2.4) {$p=1/3$};
            \fill[red, opacity=0.5] (6,-0.9) rectangle (10,-0.7);
            \node[red] at (8,-1.4) {$p=1/3$};
            \node[red] at (0.5,-2.2) {\large{\textbf{Hider}}};
            \fill[blue, opacity=0.5] (10/15,0.5) rectangle (10/15+4,0.3);
            \node[blue] at (10/15+2,1.) {$p=1/2$};
            \fill[blue, opacity=0.5] (80/15,0.5) rectangle (80/15+4,0.3);
            \node[blue] at (80/15+2,1.) {$p=1/2$};
            \node[blue] at (0.7,1.7) {\large{\textbf{Searcher}}};
            \end{tikzpicture}
        }
        \caption{$\alpha = 0.4$}
        \label{fig:game1}
    \end{subfigure}
    \hfill
    \begin{subfigure}[b]{0.49\textwidth}
        \centering
        \resizebox{\textwidth}{!}{
            \begin{tikzpicture}
            \draw[thick] (0,0) -- (10,0);
            \foreach \x/\label in {0/0, 1/0.1, 2/0.2, 3/0.3, 4/0.4, 5/0.5, 6/0.6, 7/0.7, 8/0.8, 9/0.9, 10/1}
                \draw (\x,0.1) -- (\x,-0.1) node[below] {\label};
            \fill[red, opacity=0.5] (0,-0.9) rectangle (3,-0.7);
            \node[red] at (1.5,-1.4) {$p=1/4$};
            \fill[red, opacity=0.4] (70/30,-1.9) rectangle (70/30+3,-1.7);
            \node[red] at (70/30+1.5,-2.4) {$p=1/4$};
            \fill[red, opacity=0.5] (140/30,-0.9) rectangle (140/30+3,-0.7);
            \node[red] at (140/30+1.5,-1.4) {$p=1/4$};
            \fill[red, opacity=0.5] (210/30,-1.9) rectangle (210/30+3,-1.7);
            \node[red] at (8.5,-2.4) {$p=1/4$};
            \node[red] at (0.5,-2.2) {\large{\textbf{Hider}}};
            \fill[blue, opacity=0.5] (10/40,0.5) rectangle (10/40+3,0.3);
            \node[blue] at (1/40+1.5,1.) {$p=1/3$};
            \fill[blue, opacity=0.5] (140/40,0.5) rectangle (140/40+3,0.3);
            \node[blue] at (140/40+1.5,1.) {$p=1/3$};
            \fill[blue, opacity=0.5] (270/40,0.5) rectangle (270/40+3,0.3);
            \node[blue] at (270/40+1.5,1.) {$p=1/3$};
            \node[blue] at (0.7,1.7) {\large{\textbf{Searcher}}};
            \end{tikzpicture}
        }
        \caption{$\alpha = 0.3$}
        \label{fig:game2}
    \end{subfigure}
    \caption{Optimal strategies on the unit interval with fixed interval lengths.}
    \label{fig:games}
\end{figure}

\begin{proof}[Proof of Theorem~\ref{thm:interval}]
We first show that the Hider can ensure the payoff is at least $v$. It is clear that any Searcher strategy must have non-empty intersection with at most two of the intervals in the Hider's support. Furthermore, a best response of the Searcher must contain the whole of the intersection of two adjacent intervals in the Hider's support. In fact, it is easy to see that {\em any} Searcher strategy that contains the whole of the intersection of two adjacent intervals in the Hider's support has the same expected payoff against the Hider's strategy, since the intersections all have the same length. 

Let us take one particular such Searcher strategy: the interval with point $b=0$. The payoff of this strategy is $0$ against the interval with point $a_0$, and is $2\alpha$ against all the intervals with points $a_2,\ldots,a_M$. Against the interval with point $a_1$, the payoff is $2a_1$. Therefore, the expected payoff against the Hider's proposed strategy is
\[
\left(\frac{M-1}{M+1}\right)2\alpha + \left(\frac{1}{M+1}\right)\frac{2(1-\alpha)}{M} = v.
\]
It follows that the Hider can ensure a payoff of at least $v$.

Now consider the Searcher's proposed strategy. A best response of the Hider maximizes the length of the intersection of his interval with the gaps between the intervals in the Searcher's support. Furthermore, the Searcher's interval cannot intersect with with more than one of these gaps, so a best response must contain the whole of precisely one gap. Any such strategy has the same expected payoff against the Searcher's strategy, since the gaps all have the same length. 

We take one particular choice of Hider strategy: the interval with point $a=0$. The payoff of this strategy is $2\alpha$ against all the intervals with points $b_2,\ldots,b_M$, and is $2b_1$ against the interval with point $b_1$. Therefore, the expected payoff against the Searcher's proposed strategy is
\[
\frac{M-1}{M}2\alpha + \frac{1}{M} \cdot 2\left(\frac{1+\alpha}{M+1}-\alpha \right) = v.
\]
\end{proof}

It is worth noting that as $\alpha \rightarrow 0$, the importance of the points of the unit interval decreases, and the value of the game converges to the value of the game in the case of the circle. Indeed, let us rewrite the value $v$ of the game as follows.
\[
v=\frac{ M}{M+1}2\alpha -\frac{1}{\alpha(M+1)}2\alpha^2 + \frac{2(1-\alpha)}{M(M+1)}.
\]
As $\alpha \rightarrow 0$, the coefficients $M/(M+1)$ and $-1/(\alpha(M+1))$  converge to $1$ and $-1$, respectively, and $2(1-\alpha)/(M(M+1))$ converges to $0$. So for small $\alpha$, the value $v$ approaches $2\alpha-2\alpha^2=2\alpha(1-\alpha)$, which is the value of the analogous game on the circle, as shown in Proposition~\ref{lemma:exp_payoff} (for the case $\alpha=\beta$).

\section{Discrete Alignment Games}
\label{sec:discrete}

In this section and the next, we consider variations of the game where the search space $Q$ is finite and equal to $[n] \equiv \{1,2,\ldots,n\}$, for some positive integer $n$. In this case, the cost and penalty functions are given by vectors $(c_1,\ldots,c_n)$ and $(\pi_1,\ldots,\pi_n)$ of positive reals. In particular, for $A \subseteq Q$, we take $C(A)=c(A) \equiv \sum_{j \in A} c_j$ and $\Pi(A)=\pi(A) \equiv \sum_{j \in A} \pi_j$. Thus, for fixed strategies $H$ and $S$ of the Hider and Searcher, respectively, the payoff $P(H,S)$ is given by $P(H,S)=c(S \setminus H) + \pi(H \setminus S)$.

We will consider two different variants of the game, depending on the Hider's and Searcher's feasible subsets. In Subsection~\ref{sec:any} we consider the variant of the game where the players may choose any subset of $[n]$. In Subsection~\ref{sec:fixed}, we consider the variant where the cardinality $k$ of the Hider's set is fixed, whereas the Searcher can choose any subset of $[n]$. In this case, we will see that even for $n=2$ and $k=1$, the solution is somewhat complex.

\subsection{Non-equal Costs and Penalties} \label{sec:any}

In this section, we assume that the collections $\mathcal{H}$ and $\mathcal{S}$ of the Hider's and Searcher's feasible subsets of $[n]$ are both equal to the power set $2^{[n]}$ of {\em all} subsets of $[n]$. In this case, it turns out the game has a particular elegant solution.

For a subset $A \in 2^{[n]}$, let $f(A) \equiv \sum_{B\subseteq A} \prod_{j\in B} c_j/\pi_j$.

\begin{theorem}
\label{thm:no_k}
Consider the discrete alignment game with $\mathcal{H}=\mathcal{S}=2^{[n]}$. An optimal Hider strategy is to choose a subset $H \subseteq [n]$ with probability $p_H$, given by $$ p_H := \lambda \prod_{j \in H} \frac{c_j}{\pi_j},$$

\noindent where $\lambda$ is a normalizing constant factor given by  $ \lambda = f([n])^{-1}$.

An optimal strategy for the Searcher is to choose $S \subseteq [n]$ with probability $p_{H^c}$.

The value of the game is
$$
V=\sum_{j \in [n]}\frac{f([n] \setminus \{j\})}{f([n])}\cdot c_j,
$$

\end{theorem}

\begin{proof}
Suppose the Hider plays the strategy described. Then, the probability $q_j$ that some location $j\in[n]$ is contained in the Hider's subset is
$$
q_j=\sum_{\{H\subseteq[n]:j\in A\}}p_H=\frac{c_j}{\pi_j}\cdot\frac{f([n] \setminus \{j\})}{f([n])}.
$$
The complementary probability $1-q_j$ can be written as
$$
1-q_j=\sum_{\{H\subseteq[n]:j\notin H\}}p_H=\frac{f([n] \setminus \{j\})}{f([n])}.
$$
Now let $S$ be an arbitrary Searcher strategy. The payoff is incremented by $(1-q_j)c_j$ for every $j \in S$, and by $q_j\pi_j$ for every $j \notin S$. Hence, denoting the Hider's mixed strategy by $p$, the expected payoff against $S$ is
\begin{align*}
P(p,S)&=\sum_{j\in S}(1-q_j)c_j+\sum_{\{j\}\notin S}q_j\pi_j \\
&=\sum_{j\in S}\frac{f([n] \setminus \{j\})}{f([n])}\cdot c_j+\sum_{\{j\}\notin S}\frac{c_j}{\pi_j}\cdot\frac{f([n] \setminus \{j\})}{f([n])}\cdot \pi_j\\
&=\sum_{j\in [n]}\frac{f([n] \setminus\{j\})}{f([n])}\cdot c_j 
\end{align*}

Hence, the Hider can ensure the payoff is at least $V$. By an almost identical argument, the Searcher can ensure the payoff is at most $V$, so the value of the game is $V$.\end{proof}

Note that, for the situation where $c_i=\pi_i$ for all $i$, we have $f(H)=2^{|H|}$ for all $H \subseteq [n]$, so $p_H=1/2^n$, and it is optimal for both players is to choose all subsets  with equal probability. In this case, the value of the game is $\sum_{i=1}^n c_j/2$.

\subsection{Fixed Number of Hiding Locations} \label{sec:fixed}

We now consider a slight variation of the game above, where the Hider's feasible subsets $\mathcal{H}$ are all the subsets of $[n]$ of cardinality $k$, and the Searcher's feasible subsets are, again $\mathcal{S}=2^{[n]}$. The Hider now has $\binom{n}{k}$ pure strategies, whereas the Searcher has $2^n$. 

We derive a partial solution for a specific instance where $n=2$ and $k=1$. In this case, the Hider must place a fault in location 1 or 2; therefore, his possible strategy sets are $\{1\},\{2\}$. The Searcher, despite knowing $k=1$, still chooses among all possible subsets of $[2]$: $\emptyset,\{1\},\{2\},\{1,2\}$. 

\begin{proposition}

The solution of the discrete alignment game with $\mathcal{H}=\{\{1\},\{2\}\}$ and $\mathcal{S}=2^{\{1,2\}}$ splits into four scenarios. Assume, without loss of generality, that $c_1\ge c_2$.Then 

\begin{enumerate}[(i)]
    \item if $\pi_1\pi_2\le c_1c_2$ and $\pi_1\le\pi_2$, it is optimal for the Searcher to mix between $\emptyset$ and $\{2\}$;
        \item if $\pi_1\pi_2\ge c_1c_2$ and $\pi_1\le\pi_2$, it is optimal for the Searcher to mix between $\{2\}$ and $\{1,2\}$;
    \item if $\pi_1\pi_2\le c_1c_2$ and $\pi_1\ge\pi_2$, it is optimal for the Searcher to mix between $\emptyset$ and $\{1\}$;
     \item if $\pi_1\pi_2\ge c_1c_2$ and $\pi_1\ge\pi_2$, it is optimal for the Searcher to mix between $\{1\}$ and $\{1,2\}$.
\end{enumerate}
\end{proposition}

\begin{proof}
    Let us analyze each case separately.
\begin{enumerate}[(i)]
\item First, consider the case where $\pi_1\pi_2\le c_1c_2$ and $\pi_1\le\pi_2$.
In this reduced game, the Hider's optimal strategy is to play $\{1\}$ with probability $q = \pi_2/(c_2 + \pi_2)$ and $\{2\}$ with probability $1-q$. Correspondingly, the Searcher's optimal strategy is to play $\emptyset$ with probability $p = (c_2 + \pi_1)/(c_2 + \pi_2)$ and  $\{2\}$ with probability $1-p$. The value of this reduced game is $ (c_2 + \pi_1)\pi_2/(c_2 + \pi_2)$. If the Searcher deviates to play either $\{1,2\}$ or $\{1\}$ against the Hider's optimal strategy, the resulting payoff is $ (c_2\pi_2 + c_1c_2)/(c_2 + \pi_2) \geq  (c_2 + \pi_1)\pi_2/(c_2 + \pi_2)$.  Therefore, the Searcher has no incentive to deviate from mixing between $\emptyset$ and $\{2\}$.

\item Now, consider the case where $\pi_1\pi_2\ge c_1c_2$ and $\pi_1\le\pi_2$. We consider the reduced game where the Searcher mixes between $\{2\}$ and $\{1,2\}$. To make the Searcher indifferent between these two choices, the Hider must play $\{1\}$ with probability $q = c_1/(\pi_1 + c_1)$ and $\{2\}$ with probability $1-q$. The Searcher's optimal mixed strategy is to play $\{2\}$ with probability $p = (c_1 - c_2)/(\pi_1 + c_1)$ and $\{1,2\}$ with probability $1-p$. The value of this game is $V = (c_2 + \pi_1)c_1/(\pi_1 + c_1)$.  If the Searcher deviates to play either $\emptyset$ or $\{1\}$ against the Hider's optimal strategy, the resulting payoff is $(\pi_1\pi_2 + c_1\pi_1)/(\pi_1 + c_1) \geq V$ (under the condition $c_1c_2 \leq \pi_1\pi_2$). Thus, the Searcher has no incentive to deviate from mixing between $\{2\}$ and $\{1,2\}$.  

\item[(iii) \& (iv)] For the other two cases, we can simply rearrange the rows and the proof follows as before. This completes the proof.
\end{enumerate}
\end{proof}

Finding a general, closed form solution for the game in the case that one player must choose a subset of fixed cardinality and the other player can choose any subset seems intractable. For this reason, we will make the simplifying assumption of equal costs and penalties for the remainder of the paper.

\section{The Discrete Alignment Game with Equal Costs and Penalties} \label{sec:equal}

In this section we consider the discrete alignment game in the particular case that $c_i=\pi_i$ for all $i \in [n]$, so that the payoff becomes the weighted symmetric difference $P(H,S) =  c(S \Delta H) \equiv c((S \setminus H) \cup (H \setminus S))$. Note that $P(H,S)=P(S,H)$ as $S \Delta H = H \Delta S$. For arbitrary Hider and Searcher strategy sets $\mathcal{H},\mathcal{S} \subseteq 2^{[n]}$, we denote the game $\Gamma(\mathcal{H},\mathcal{S})$.

A mixed Hider strategy is given by a probability distribution $h: \mathcal{H} \rightarrow [0,1]$, where $h(H)$ is the probability of choosing a set $H$. Similarly, a mixed Searcher strategy is given by a probability distribution $s:\mathcal{S}
\rightarrow [0,1]$ where $s(S)$ is the probability of choosing a set $S$. 

Before we turn to some specific games, we establish a general result. For a collection of subsets $\mathcal{A} \subseteq 2^{[n]}$, write $\mathcal A^\#$ for the set $\{A^c: A \in \mathcal{A} \}$ of complements of sets in $\mathcal{A}$. For example, if we take $\mathcal{A}$ to be the set $[n]^{(k)}$ of subsets of $[n]$ of cardinality $k$, then $\mathcal{A}^\#$ is the set $[n]^{(n-k)}$ of subsets of $[n]$ of cardinality $n-k$. Clearly, for any $\mathcal{A} \subseteq 2^{[n]}$, we have $(A^\#)^\#=A$.

\begin{lemma}\label{lem:sym}
Suppose that $h$ and $s$ are optimal mixed Hider and Searcher strategies, respectively, for  the game $\Gamma( \mathcal{H},\mathcal{S})$, and that the value of the game is $V$. Let $h^c$ and $s^c$ be strategies defined on subsets in $\mathcal{H}^\#$ and $\mathcal{S}^\#$ respectively, given by $h^c(H^c) = h(H)$ and $s^c(S^c) = s(S)$ for $H \in \mathcal{H}$ and $S \in \mathcal{S}$. Then

\begin{enumerate}[(i)]
\item For the game $\Gamma(\mathcal{H}^\#,\mathcal{S}^\#)$, the strategy $h^c$ is optimal for the Hider and the strategy $s^c$ is optimal for the Searcher. The value of the game is $V$.
    \item For the game $\Gamma(\mathcal{S}^\#,\mathcal{H})$, the strategy $s^c$ is optimal for the Hider and $h$ is optimal for the Searcher. The value of the game is $c([n])-V$.
    \end{enumerate}
\end{lemma}

\begin{proof}
For part (i), first note that for any $A,B \subseteq [n]$, we have $A \triangle B \equiv A^c \triangle B^c$. Hence, for any fixed Hider pure strategy $H^c \in \mathcal{H}^\#$, the strategy $s^c$ ensures a payoff in the game $\Gamma(\mathcal{H}^\#,\mathcal{S}^\#)$ of 
\[
P(H^c,s^c) = \sum_{S^c \in \mathcal{S}^\#} s^c(S^c) c(H^c \triangle S^c) = \sum_{S \in \mathcal{S}
} s(S) c(H \triangle S) = P(H,s) \le V,
\]
since $s$ is optimal in $\Gamma(\mathcal{H},\mathcal{S})$. Similarly, for a fixed Searcher pure strategy $S^c \in \mathcal{S}^\#$, the strategy $h^c$ ensures a payoff in the game $\Gamma(\mathcal{H}^\#,\mathcal{S}^\#)$ of 
\[
P(h^c,S^c) = \sum_{H^c \in \mathcal{H}^\#} h^c(H^c) c(H^c \triangle S^c) = \sum_{H \in \mathcal{H}} h(H) c(H \triangle S) = P(h,S) \ge V,
\]
since $h$ is optimal in $\Gamma(\mathcal{H},\mathcal{S})$. The proves part (i) of the lemma.

For part (ii), we first note that for any $A,B \subseteq [n]$, we have $c(A^c \triangle B) = c([n]) - c(A \triangle B)$. Therefore, if the Hider uses the strategy $s^c$ in the game $\Gamma(\mathcal{S}^\#,\mathcal{H})$ against an arbitrary Searcher strategy $H \in \mathcal{H}$, the expected payoff is
\begin{align*}
    P(s^c,H) &= \sum_{S^c \in \mathcal{S}^\#} s^c(S^c) c(S^c,H)\\ 
    &= \sum_{S \in \mathcal{S}} s(S) (c([n])-c(S,H)) \\
    &= c([n])-P(s,H) \\
    & \ge c([n])-V,
\end{align*}
by the optimality of $s$ in $\Gamma(\mathcal H, \mathcal S)$. Hence, the value of the game $\Gamma(\mathcal{S}^\#,\mathcal{H})$ is at least $c([n])-V$.

Similarly, if the Searcher uses the strategy $h$ in the game $\Gamma(\mathcal{S}^\#,\mathcal{H})$ against an arbitrary Hider strategy $S^c \in \mathcal{S}^\#$, the expected payoff is
\begin{align*}
    P(S^c,h) &= \sum_{H \in \mathcal{H}} h(H) c(S^c,H)\\ 
    &= \sum_{H \in \mathcal{H}} h(H) (c([n])-c(S,H)) \\
    &= c([n])-P(S,h) \\
    & \le c([n])-V,
\end{align*}
by the optimality of $h$ in $\Gamma(\mathcal H, \mathcal S)$. Hence, the value of the game $\Gamma(\mathcal{S}^\#,\mathcal{H})$ is at most $c([n])-V$, and we have equality. This completes the proof.
\end{proof}

We will see the usefulness of Lemma~\ref{lem:sym} in the next subsection, when we solve the discrete alignment game in the case where one player must choose a subset of fixed cardinality $k$, and the other player can choose any subset. It follows from part (i) of Lemma~\ref{lem:sym} that if we can solve the game for $k \le n/2$, then the solution for $k \ge n/2$ follows. It follows from part (ii) of the lemma that if we can solve the game when one player's subset has fixed cardinality, then the solution when the other player's subset has fixed cardinality automatically follows.

\subsection{Fixed Number of Searches or Hiding Locations}

Noting that we have already solved the game with equal costs and penalties in the special case of $\mathcal{H}=\mathcal{S}=2^{[n]}$ in the previous section, we start our analysis by examining the case where the Hider chooses a fixed number $k$ of locations and the Searcher can choose any number of locations. In other words, $\mathcal{H}$ is the set $[n]^{(k)}$ of all subsets of $[n]$ of cardinality $k$, and $\mathcal{S}=2^{[n]}$. 

We assume, without loss of generality that $c_1 \ge c_2 \ge \cdots \ge c_n \ge 0$. Furthermore, we denote the value of the game by $V_k$, for each $k=0,1\ldots,n$. Our next theorem gives a full solution.

\begin{theorem}
\label{thm:k}

Consider the alignment game $\Gamma([n]^{(k)},2^{[n]})$. The value $V_k$ of the game is given by
\[
V_k = \begin{cases}
    c([2k])/2~\text{ if } k \le n/2, \\
    c([2(n-k)])/2~\text{ if } k \ge n/2.
\end{cases}
\]
The optimal strategies depend on $k$ as follows. First suppose $k \le n/2$. Let $\tilde{H}$ be any fixed subset of $[2k]$ of cardinality $k$.  It is optimal for the Hider to use the strategy $h_k$ that chooses $\tilde{H}$ and $[2k] \setminus \tilde{H}$, each with probability $1/2$. It is optimal for the Searcher to use the strategy $s_k$ that chooses the empty set $\emptyset$ with probability 
\[
p_0 \equiv \frac{1}{2} + \frac{c_{2k}}{2 c_{1}},
\]
and to choose the set $[j]$ for $j=1,\ldots,2k-1$ with probability
\[
p_j \equiv \frac{c_{2k}}{2c_{j+1}} - \frac{c_{2k}}{2c_j}.
\]
Now suppose $k \ge n/2$. Then the strategy $h_k$, defined as $h_{n-k}^c$ is optimal for the Hider and the strategy $s_k$, defined as $s_{n-k}^c$ is optimal for the Searcher. 

\end{theorem}
\begin{proof}
First suppose that $k \le n/2$. It is easy to verify that the Searcher strategy $s$ is well defined, since $\sum_{j=0}^{2k-1}p_j$ is telescopic and reduces to $1/2 + c_{2k}/(2c_{2k})=1$.

Let $S$ be any fixed Searcher strategy and write $S=A \cup B$, where $A \subseteq [2k]$ and $B \subseteq [2k]^c$. The payoff against $h_k$ is
\begin{align*}
P(h_k,S) &= \frac{1}{2} c(\tilde{H} \triangle S) + \frac{1}{2} c(([2k] \setminus \tilde{H}) \triangle S) \\
&= \frac{1}{2} \left(c(\tilde{H} \triangle A) + c(([2k] \setminus \tilde{H}) \triangle A)\right) + c(B) \\
&= c([2k])/2 + c(B) \\
&\ge c([2k])/2.
\end{align*}
So the strategy $h_k$ guarantees an expected payoff of at least $ c([2k])/2$ and we have ${V_k \ge  c([2k])/2}$.

Now suppose $H$ is any fixed Hider strategy and write $H = C \cup D$, where $C \subseteq [2k]$ and $D \subseteq [2k]^c$. Note that any location $j \in C$ is missed by the Searcher with probability ${\sum_{i=0}^{j-1} p_i = 1/2 + c_{2k}/(2c_j)}$ and any location $j \in [2k] \setminus C$ is chosen by the Searcher with probability $\sum_{i=j}^{2k-1} p_i = 1/2 - c_{2k}/(2c_j)$. Any $j \in D$ is missed by the Searcher with probability~$1$. Using these observations and $c(D) \le |D|c_{2k}$,
\begin{align*}
P(H,s_k) &= c(D) + \sum_{j \in C} c_j \sum_{i=0}^{j-1} p_i + \sum_{j \in [2k] \setminus C} c_j \sum_{i=j}^{2k-1} p_i \\
& \le |D|c_{2k} + \sum_{j \in C} c_j \left(\frac{1}{2} + \frac{c_{2k}}{2c_j}\right) + \sum_{j \in [2k] \setminus C} c_j \left( \frac{1}{2} - \frac{c_{2k}}{2c_j} \right) \\
&=|D|c_{2k} + \frac{c([2k])}{2} + \frac{c_{2k}}{2} (|C|-|[2k]\setminus C|) \\
&=\frac{c([2k])}{2} + \frac{c_{2k}}{2} (2|D| + |C|-(2k-|C|)) \\
&= \frac{c([2k])}{2}, 
\end{align*}
since $|C|+|D| = k$. So the strategy $s_k$ guarantees an expected payoff of at most $ c([2k])/2$ and we have $V_k \le  c([2k])/2$. 

Now suppose that $k \ge n/2$. Then we may express the game $\Gamma([n]^{(k)},2^{[n]})$ as $\Gamma(([n]^{(n-k)})^\#,(2^{[n]})^\#)$. By Lemma~\ref{lem:sym}, part (i), the strategy $h_{n-k}^c$ is optimal for the Hider, and the strategy $s_{n-k}^c$ is optimal for the Searcher. The value of the game is $V_{n-k}=c([2(n-k)])/2$.
\end{proof}

\begin{figure}[t]
    \centering
    \begin{tikzpicture}[scale=1, transform shape,
        box/.style={draw, rectangle, minimum height=1cm, minimum width=1cm, align=center}]
        \node[box] (b1) at (0,0) {\shortstack{$c_1$}};
        \node[box] (b2) at (1.7,0) {\shortstack{$c_2$}};
        \node[box] (b3) at (3.4,0) {\shortstack{$c_3$}};
        \node[box] (b4) at (5.1,0) {\shortstack{$c_4$}};
        \node[box] (b5) at (6.8,0) {\shortstack{$c_5$}};

        \begin{scope}[on background layer]
            \node[draw=red, ultra thick, dashed, fit=(b1) (b2) (b3) (b4), inner sep=8pt, rounded corners] (hider_region) {};
            \node[draw=blue, ultra thick, dashed, fit=(b1) (b2) (b3), inner sep=5pt, rounded corners] (searcher_region) {};
        \end{scope}

        \node[above=2pt of hider_region.north, align=center, color=red] {Hider considers the $[4]$ set split in two};
        \node[below=2pt of hider_region.south, align=left, color=blue] {Searcher Base Sets: ($\emptyset, [1],[2],[3]$)};

        \node[above right=-0.15cm and 0.2cm of b5.south east, align=left, font=\small, color=grey] {Outside $[4]$, \\ never in play.};

    \end{tikzpicture}
    \caption{Visualization of relevant location sets for optimal strategies in $G_k$ with $n=5, k=2$ ($2k=4$). The red dashed box highlights the set $[2k]$ where the Hider concentrates his strategy. The blue dashed box highlights the locations $\{1, \ldots, 2k-1\}$ that form the maximal extent of the Searcher's base sets ($\emptyset, [1], \ldots, [2k-1]$). Note that the cheapest location is never visited, and the Hider optimal strategy includes a location that will not be included in the Searcher's strategy.}
    \label{fig:strategy_regions}
\end{figure}
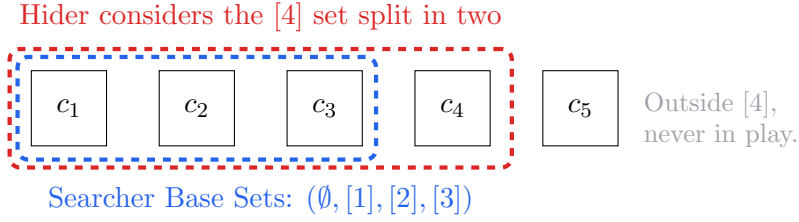

We can visualize a simple case where $n=5$ and $k=2$ in Figure \ref{fig:strategy_regions}. Note that the Hider will consider hiding in the most expensive $2k=4$ boxes. The Searcher would then choose to inspect $\emptyset$ with probability $\frac{1}{2}+\frac{c_4}{2c_3}$; $\{1\}$ with probability $\frac{c_4}{2c_2}+\frac{c_4}{2c_1}$; $\{1,2\}$ with probability $\frac{c_4}{2c_3}+\frac{c_4}{2c_2}$; $\{1,2,3\}$ with probability $\frac{c_4}{2c_4}+\frac{c_4}{2c_3}$.

Note that if $n$ is even and $k=n/2$ (so that the value of the game is $c([n])/2$), the Searcher has another simple optimal strategy that makes an equiprobable choice between any fixed $S \subseteq [n]$ and its complement. Against any Hider strategy $H$, the expected cost of such a strategy is $(1/2)c(S \triangle H) + (1/2)c(S^c \triangle H) = c([n])/2$.

Now consider the variant of the game where the Searcher must choose $k$ locations and the Hider can choose any number of locations, so that $\mathcal{H}=2^{[n]}$ and $\mathcal{S}=[n]^{(k)}$. The following theorem follows immediately from Lemma~\ref{lem:sym}.

\begin{theorem}
    For the discrete alignment game  $\Gamma(2^{[n]},[n]^{(k)})$, it is optimal for the Hider to use the strategy $s_k^c$ and for the Searcher to use strategy $h_k$. The value of the game is $c([n])- V_{k}$.
\end{theorem}

\begin{proof}
    By Theorem~\ref{thm:k}, the strategies $h_k$ and $s_k$ are optimal for the Hider and Searcher, respectively in the game $\Gamma([n]^{(k)},2^{[n]})$, and the value of the game is $V_k$. By Lemma~\ref{lem:sym}, part~(ii), the value of the game $\Gamma(2^{[n]},[n]^{(k)})=\Gamma(2^{[n]\#}, [n]^{(k)})$ is equal to $c([n])-V_k$, and the strategies $s_k^c$ and $h_k$ are optimal for the Hider and Searcher, respectively.
\end{proof}

As one would expect, the value $c([n])-V_k$ of the game $\Gamma(2^{[n]},[n]^{(k)})$ in which the Searcher is restricted to choosing $k$ locations is higher than the value $V_k$ of the game $\Gamma([n]^{(k)},2^{[n]})$ in which the Hider is restricted to choosing $k$ locations, since (for $k \le n/2$)
\[
(c([n])-V_k)-V_k = c([n])-c([2k]) \ge 0,
\]
and similarly for $k \ge n/2$. This is because restricting the Searcher's strategy set and expanding the Hider's strategy set can only benefit the Hider.

It is also worth pointing out that for the game $\Gamma(2^{[n]},2^{[n]})$, where both players are unrestricted, the value is $c([n])/2$, as pointed out at the end of Subsection~\ref{sec:any}. This value is lower that of the game $\Gamma(2^{[n]},[n]^{(k)})$ (when the Searcher is restricted) and higher than that of the game $\Gamma([n]^{(k)},2^{[n]})$ (when the Hider is restricted).

Evidently, for the game $\Gamma([n]^{(k)},[n]^{(k)})$, where {\em both} players are restricted to choosing subsets of size $k$, this is an disadvantage to the Hider compared to $\Gamma(2^{[n]},[n]^{(k)})$ and a disadvantage to the Searcher compared to  $\Gamma([n]^{(k)},2^{[n]})$. Thus, the value of the game $\Gamma([n]^{(k)},[n]^{(k)})$ must lie between $V_k$ and $c([n])-V_k$.

\subsection{Fixed Number of Searches and Hiding Locations}
\label{sec:fixed_searchers_and_hiding_locations}

Finally, we turn to the variant $\Gamma([n]^{(k)},[n]^{(k)})$ of the discrete alignment game where both players choose a subset of fixed size. This variant turns out to be the most complex version of the game, and we present here a solution only for $k=1$, leaving other cases for future work. 

For $k=1$, the payoff structure for the game is particularly simple: if the Searcher chooses some $i$ and the Hider chooses some $j$, then $P(i,j)$ is equal to $c_i + c_j$ if $i \neq j$, otherwise $P(i,j)=0$ if~$i=j$. As before, we assume, without loss of generality, that $c_1 \geq c_2 \geq \ldots \geq c_n$.

If $n=2$, the game has a simple solution: both players choose each strategy with probability $1/2$, and the value of the game is $(c_1+c_2)/2$. In the case that $c_1=c_2=1$, this game is equivalent to Matching Pennies, as mentioned in the Introduction. In this game, Players 1 and 2 both choose heads or tails. If they match, Player 1 gets a payoff of 1; if not he gets a payoff of -1.

For the remainder of this subsection we will assume that $n \ge 3$. We first define the strategies $h_M$ and $s_M$ for the Hider and Searcher, respectively, that we will subsequently show are optimal. 

Let $M=2,\ldots,n-1$ be maximal such that
\begin{align}
\sum_{i=1}^{M} \frac{1}{c_i} \ge \frac{M-2}{c_n}. \label{eq:M}
\end{align}

Note that~(\ref{eq:M}) holds for $M=2$, since
\[
\sum_{i=1}^2 \frac{1}{c_i} \ge 0 = \frac{2-2}{c_n}.
\]
Therefore, $M$ is well defined.

There are two cases.

\textbf{Case 1: $M=n-1$.} Let $h_{n-1}$ denote the Hider strategy that chooses $j=1,\ldots,n$ with probability $p_j$, given by
\[
p_j = \frac{1}{2} - \frac{(n-2)}{2c_j\sum_{i=1}^n 1/c_i}.
\]
Note that for any $j=1,\ldots,n$, we have
    \[
\sum_{i=1}^n \frac{1}{c_i} = \frac{1}{c_n} + \sum_{i=1}^M \frac{1}{c_i} \ge \frac{1}{c_j}+\frac{M-2}{c_n}= \frac{n-2}{c_n}, 
    \]
where the inequality follows from the definition of $M$ and the fact that $c_n \le c_j$. It follows that for each $j=1,\ldots,n$, the probability $p_j$ is non-negative (and it is clearly less than $1$).

The Searcher strategy $s_{n-1}$ is defined identically to $h_{n-1}$; that is, $j$ is chosen with probability $p_j$, as defined above.

\textbf{Case 2: $M \le n-2$.} Let $h_M$ denote the Hider strategy that, for $j=1,\ldots,M$ chooses $j$ with probability $p_j$, given by
\[
p_j= \frac{1}{2} - \frac{c_n}{2c_j},
\]
and chooses $M+1$ with probability 
\[
p_{M+1} = 1- \sum_{j=1}^{M} p_j =- \frac{M-2}{2} + \sum_{j=1}^{M} \frac{c_n}{2c_j}.
\]
It is clear that for $j=1,\ldots,M$, the probability $p_j$ is well defined, and the non-negativity of $p_{M+1}$ follows from the definition of $M$.

Let $s_M$ denote the Searcher strategy that, for $i=1,\ldots,M$ chooses $i$ with probability $q_j$, given by
\[
q_j= \frac{1}{2} - \frac{c_{M+1}}{2c_j},
\]
and chooses $n$ with probability 
\[
q_n = 1- \sum_{j=1}^{M} q_j = - \frac{M-2}{2} + \sum_{j=1}^{M} \frac{c_{M+1}}{2c_j}.
\]

\begin{theorem}
The strategies $h_M$ and $s_M$ are optimal. If $M=n-1$ (Case 1), then the value of the game is
\[
V_1 \equiv \frac{1}{2} \sum_{i=1}^n c_i - \frac{(n-2)^2}{2 \sum_{i=1}^n 1/c_i}.
\]
If $M \le n-2$ (Case 2), then $h_M$ and $s_M$ are optimal and the value of the game is 
\[
V_2 \equiv \frac{-(M-2)(c_n+c_{M+1})}{2} + \frac{1}{2}\sum_{i=1}^{M} \left( c_i + \frac{ c_{M+1} c_n}{c_i} \right).
\]
\end{theorem}
\begin{proof}
    We begin with Case 1, when $M=n-1$. It is easy to check that if the Hider plays $h_{n-1}$ against an arbitrary strategy $k$ of the Searcher, the expected payoff is 
    \begin{align*}
        P(h_{n-1},k) & =\sum_{j \neq k} p_j(c_j + c_k) \\
        & = (1-p_k)c_k + \sum_{j \neq k} p_j c_j \\
        & = \frac{c_k}{2} + \frac{n-2}{2 \sum_{i=1}^n 1/c_i} + \sum_{j \neq k} \left(\frac{c_j}{2} - \frac{n-2}{2 \sum_{i=1}^n 1/c_i} \right) \\
        & =V_1 .
    \end{align*}
    
    Since $s_{n-1}$ is equal to $h_{n-1}$, the same payoff is guaranteed by $s_{n-1}$ against any Hider strategy, so this payoff is equal to the value of the game, and $h_{n-1}$ and $s_{n-1}$ are optimal.

    Turning to Case 2, we first consider the Hider strategy $h_M$. Against the Searcher strategy $n$, the expected payoff is
    \begin{align*}
        P(h_M, n) &= \sum_{i \le M+1} p_i (c_i +c_n) \\
        &= c_n + \left(- \frac{M-2}{2} + \sum_{i=1}^{M} \frac{c_n}{2c_i} \right)c_{M+1} + \sum_{i \le M} \left( \frac{1}{2} - \frac{c_n}{2c_i} \right)c_i \\
        &=V_2.
    \end{align*}
    If the Searcher plays some strategy $j \le M$, then the expected payoff is
    \begin{align*}
    P(h_M, j) & = \sum_{i \le M+1, i\neq j} p_i (c_i +c_j) \\
    &= (1-2p_j)c_j - c_n + \sum_{i \le M+1} p_i(c_i +c_n)\\
    &= P(h_M, n)+2c_{j} \left( \frac{1}{2} - \frac{c_n}{2 c_{j}} - p_{j}\right) \\
    & = V_2,
    \end{align*}
by the definition of $p_j$ and the earlier calculation of $P(h_M,n)$.
    
If the Searcher chooses some $j$ with $M+2 \le j \le n-1$, then the payoff is
    \[
    \sum_{i \le M+1} p_i(c_i+c_j) = c_j-c_n+ \sum_{i \le M+1} p_i(c_i+c_n) \ge P(h_M,n)=V_2.
    \]
    For the remaining case that the Searcher plays $M+1$, we first show that $p_{M+1} \le 1/2 - c_n/(2c_{M+1})$. Indeed,
    \begin{align*}
      \left( \frac{1}{2} - \frac{c_n}{2c_{M+1}} \right)- p_{M+1} &= \frac{c_n}{2} \left(\frac{1}{c_n} - \frac{1}{c_{M+1}} + \frac{M-2}{c_n} - \sum_{i=1}^{M} \frac{1}{c_i} \right) \\
      & = \frac{c_n}{2} \left( \frac{M-1}{c_n} - \sum_{i=1}^{M+1} \frac{1}{c_i} \right) \\
      & \ge 0,
    \end{align*}
    by definition of $M$. We now calculate
    \begin{align*}
        P(h_M,M+1) &= \sum_{i=1}^M p_i(c_i+c_{M+1}) \\
        & = (1-2p_{M+1})c_{M+1} - c_{n} + \sum_{i =1}^{M+1} p_i(c_i +c_n) \\
        & = V_2 +2c_{M+1} \left( \frac{1}{2} - \frac{c_n}{2 c_{M+1}} - p_{M+1}\right) \\
        & \ge V_2.
    \end{align*}
We now consider the Searcher strategy $s_M$. Against the Hider strategy $M+1$, the expected payoff is
\begin{align*}
    P(M+1,s_M) &= q_n (c_{M+1}+c_n) + \sum_{i=1}^M q_i(c_{M+1}+c_i) \\
    &= c_{M+1}+\left(- \frac{M-2}{2} + \sum_{i=1}^{M} \frac{c_{M+1}}{2c_i} \right) c_n + \sum_{i=1}^M \left( \frac{1}{2} - \frac{c_{M+1}}{2c_i} \right) c_i \\
    &=V_2.
\end{align*}
If the Hider plays some $j=1,\ldots,M$, then the expected payoff against $s_M$ is
\begin{align*}
    P(j, s_M) &= q_n(c_j + c_n) + \sum_{i\le M, i\neq j} q_i(c_i+c_j) \\
    &= P(M+1,s_M) + 2c_j \left(\frac{1}{2} - \frac{c_{M+1}}{2c_j} - q_j \right) \\
    &= V_2.
\end{align*}
If the Hider plays some strategy $j$ with $M+2 \le j \le n-1$, then the expected payoff against $s_M$ is
\[
P(j, s_M) = q_n(c_j+c_n)+\sum_{i = 1}^M q_i(c_i+c_j) = c_j - c_{M+1} + P(M+1,s_M) \le P(M+1,s_M)=V_2.
\]
Finally, for the case the Hider plays strategy $n$ against $s_M$, we first show that $q_n \ge 1/2 - c_{M+1}/(2 c_n)$. Indeed,
\begin{align*}
    q_n - \left(\frac{1}{2} - \frac{c_{M+1}}{2c_n} \right) 
    &= \frac{c_{M+1}}{2} \left( -\frac{M-2}{c_{M+1}} + \sum_{i=1}^{M} \frac{1}{c_i} - \frac{1}{c_{M+1}} + \frac{1}{c_n} \right) \\
    &\ge \frac{c_{M+1}}{2} \left( -\frac{M-1}{c_{M+1}} + \frac{M-2}{c_n} + \frac{1}{c_n}\right) \\
    & \ge 0,
\end{align*}
where the first inequality follows from the definition of $M$ and then second follows from $c_n \le c_{M+1}$.

We now calculate
\begin{align*}
    P(n, s_M) &= \sum_{i = 1}^M q_i(c_i+c_n) \\
    &= P(M+1,s_M) +2c_n \left(\frac{1}{2} - \frac{c_{M+1}}{2 c_n} - q_n \right) \\
    & \le V_2,
\end{align*}
and the proof is complete.
\end{proof}

\section{Conclusions}

This paper introduces alignment games, a class of zero-sum games analyzing strategic problems where decision-makers face dual risks from miscalibrated intervention: commission errors (unnecessary action) and omission errors (necessary action forgone). Motivated by operational contexts including medical diagnostics, resource allocation, and monitoring problems, we develop a theoretical framework that abstracts the essential mathematical structure of such problems. alignment games extend traditional search game models by explicitly incorporating costs for both error types, providing a mathematical foundation for analyzing the trade-off between comprehensive coverage and precise targeting.

We derive closed-form equilibrium solutions across multiple settings. For continuous domains (Table \ref{tab:cont_game_summary}), when both players freely choose arc lengths on the unit circle, equilibrium involves $\alpha^* = c/(c+\pi)$ and $\beta^* = \pi/(c+\pi)$ with game value $c\pi/(c+\pi)$.  When one player's length is fixed, optimal strategies exhibit threshold behaviors at critical values $\alpha^*$ and $\beta^*$. For the unit interval under equal costs, strategic structures vary: when lengths are variable, optimal play often involves equiprobable choices between specific symmetric subintervals like $[0,1/2]$ and $[1/2,1]$. When interval lengths are fixed for both players, optimal strategies involve the Hider mixing over a minimal set of covering intervals, while the Searcher mixes over a maximal set of non-overlapping intervals.

\begin{table*}[t]
    \centering
    \caption{Summary of optimal strategies and game values in continuous alignment games}
    \label{tab:cont_game_summary}
    \footnotesize 
    \begin{tabular}{>{\RaggedRight}p{3.4cm} >{\RaggedRight}p{4.2cm} >{\RaggedRight}p{4.cm} >{\RaggedRight}p{3.5cm}}
    \toprule
        \textbf{Scenario} & \textbf{Hider's Strategy} & \textbf{Searcher's Strategy} & \textbf{Value} \\
        \midrule
        \rowcolor[gray]{0.9} 
        \multicolumn{4}{l}{\textbf{Unit Circle}}  \\
        Both choose lengths& Choose length $\frac{c}{c+\pi}$ uniformly. & Choose length $\frac{\pi}{c+\pi}$ uniformly. & $\frac{c\pi}{c+\pi}$\\
        \hline
        Hider fixed $\alpha$, Searcher chooses $\beta$ & Fixed arc length $\alpha$, start uniformly. & Choose $\beta=0$ if $\alpha$ is small; $\beta=1$ if $\alpha$ is large. & $(1-\alpha)c$ or $\alpha\pi$ \\
        \hline
        Searcher fixed $\beta$, Hider chooses $\alpha$ & Choose $\alpha=1$ if $\beta$ is small; $\alpha=0$ if $\beta$ is large. & Fixed arc length $\beta$, start uniformly. & Analogous to above\\
        \hline
        Both fixed $\alpha$ and $\beta$ & Choose start uniformly. & Choose start uniformly. & Expected symmetric difference \\
        \midrule
        \rowcolor[gray]{0.9} 
        \multicolumn{4}{l}{\textbf{Unit Interval (Equal Costs/Penalties)}} \\
        Both choose lengths & Mix equally: $[0,1/2]$ or $[1/2,1]$. & Mix equally: $[0,1/2]$ or $[1/2,1]$. & $1/2$ \\
        \hline
        Hider fixed $\alpha$, Searcher chooses $\beta$ & Mix equally: $[0,\alpha]$ or $[1-\alpha,1]$. & Choose empty set if $\alpha\leq1/2$; full interval if $\alpha\geq1/2$. & $\alpha$ or $1-\alpha$ \\
        \hline
        Searcher fixed $\beta$, Hider chooses $\alpha$ & Choose full interval if $\beta\leq1/2$; empty set if $\beta\geq1/2$. & Mix equally: $[0,\beta]$ or $[1-\beta,1]$. & $1-\beta$ or $\beta$ \\
        \hline
        Both fixed $\alpha=\beta$ & Minimal covering strategy. & Maximal non-overlapping strategy. & Depends on $\alpha$. \\
        \bottomrule
    \end{tabular}
\end{table*}

The discrete setting (summarized in Table~\ref{tab:disc_game_summary}) also yields comprehensive equilibrium characterizations across various structural assumptions. 
When players can choose any subset of $[n]$ with heterogeneous costs and penalties, optimal mixed strategies for both Hider and Searcher assign probabilities $p_H = \lambda \prod_{j \in H} (c_j/\pi_j)$, directly linking choice probabilities to element-specific cost-penalty ratios. 

When costs are equal but cardinality constraints are imposed, distinct strategic patterns emerge. For instance, if the Hider must choose exactly $k$ elements (and $k \le n/2$), optimal play remarkably concentrates on the $2k$ highest-cost locations; the Hider mixes over complementary $k$-subsets within this restricted domain, while the Searcher employs a calibrated mixture over nested subsets $\{[j] : j = 0, \ldots, 2k-1\}$. 
Symmetrical results, derivable through structural properties relating games via set complementation, extend these solutions to cases where $k > n/2$ or when the Searcher faces the cardinality constraint instead of the Hider. 
Finally, when both players are restricted to choosing exactly $k$ elements (e.g., $k=1$), optimal strategies involve threshold-based constructions contingent on the specific values and relative magnitudes of the sorted location costs.

\begin{table*}[t]
    \centering
    \caption{Summary of optimal strategies and game values in discrete alignment games}
    \label{tab:disc_game_summary}
    \footnotesize 
    \begin{tabular}{>{\RaggedRight}p{3.4cm} >{\RaggedRight}p{4.2cm} >{\RaggedRight}p{4.cm} >{\RaggedRight}p{3.5cm}}
    \toprule
        \textbf{Scenario} & \textbf{Hider's Strategy} & \textbf{Searcher's Strategy} & \textbf{Value} \\
        \midrule
        \rowcolor[gray]{0.9}
        \multicolumn{4}{l}{\textbf{Non-equal Costs/Penalties}} \\
        Both choose from all subsets of $[n]$ (the power set $2^{[n]}$) & Chooses subset $H$ with prob. $\propto \prod_{j \in H} (c_j/\pi_j)$. & Chooses subset $S$ with complementary prob. to the Hider. & Weighted sum of costs. \\
        \midrule
        \rowcolor[gray]{0.9}
        \multicolumn{4}{l}{\textbf{Equal Costs and Penalties}} \\
        Hider chooses $k$ locations, Searcher any & Mix over a $k$-subset and its complement within the top $2k$-cost items. & Mix over nested subsets of the top $2k-1$ cost items. & \vspace{-1cm} \[\begin{cases}
        \frac{c([2k])}{2}~\text{ if } k \le n/2, \\
        \frac{c([2(n-k)])}{2}~\text{ if } k \ge n/2.
            \end{cases}
            \] \\
        \hline
        Searcher chooses $k$ locations, Hider any & Uses the Searcher's strategy from the game above. & Uses the Hider's strategy from the game above. & Total cost minus dual game's value. \\
        \hline
        Both choose $k=1$ ($n=2$) & Mix equally. & Mix equally. & $(c_1+c_2)/2$ \\
        \hline
        Both choose $k=1$ ($n\geq3$) & Mix top $M+1$ cost items, where $M$ is a cost-dependent threshold. & Mix top $M$ cost items and the cheapest item. & Depends on the cost threshold $M$. \\
        \bottomrule
    \end{tabular}
\end{table*}

The theoretical analysis reveals several mathematically interesting properties. With unequal error costs, optimal strategies exhibit proportionality to cost-penalty ratios. With equal costs, the games exhibit threshold behaviors where optimal strategies change discontinuously as parameters cross critical values. These threshold effects demonstrate how the mathematical structure of the problem creates natural phase transitions in equilibrium behavior.

From a theoretical perspective, alignment games connect to established concepts in game theory and statistics. The framework recasts the Type I/Type II error trade-off from statistical hypothesis testing into a game-theoretic setting. Classical models emerge as special cases: Matching Pennies appears when players choose single elements from a two-element set, demonstrating how the framework generalizes fundamental zero-sum interactions.

The mathematical structures we developed provide a theoretical foundation for analyzing strategic problems involving intervention under uncertainty. While our models necessarily abstract from the complexity of real-world applications, they isolate and solve the fundamental mathematical structures underlying such decisions. The partial solutions for discrete games with equal cardinality constraints point toward deeper mathematical structures requiring further investigation. Natural theoretical extensions include providing complete solutions for discrete alignment games where both players must choose $k > 1$ elements, likely generalizing the threshold-based constructions observed for $k=1$, and extending the continuous framework to scenarios with differing fixed interval lengths ($\alpha \neq \beta$) under heterogeneous costs. Alignment games contribute to the game-theoretic literature by extending classical frameworks to address a fundamental class of strategic problems characterized by the tension between action and inaction under uncertainty.

\section*{Acknowledgments}
This material is based upon work supported by the Air Force Office of
Scientific Research under award number FA9550-23-1-0556.

\bibliographystyle{apalike}
\bibliography{thesis} 

\end{document}